\documentclass[11pt]{article}
\usepackage[english]{babel}
\usepackage{amsmath}
\usepackage{amsthm}
\usepackage{amssymb}
\usepackage[dvips]{color}
\newtheorem{theorem}{Theorem}[section]
\newtheorem{definition}[theorem]{Definition}
\newtheorem{lemma}[theorem]{Lemma}
\newtheorem{proposition}[theorem]{Proposition}
\newtheorem{corol}[theorem]{Corollary}
\newtheorem{remark}[theorem]{Remark}

\newcommand{\pn}{\par\noindent}

\newcommand{\ov}{\overline}

\newcommand{\rr}{\mathbb{R}}

\newcommand{\hh}{\mathbb{H}}

\newcommand{\pp}{\partial}

\title{\bf  A Cauchy kernel for slice regular functions}

\author{Fabrizio Colombo\\ \normalsize Dipartimento di Matematica, Politecnico di
Milano\\ \normalsize Via Bonardi, 9, 20133 Milano, Italy,
fabrizio.colombo@polimi.it \and Graziano Gentili\\
\normalsize Dipartimento di Matematica, Universit\`a di Firenze, \\
\normalsize Viale Morgagni, 67 A, Firenze, Italy,
gentili@math.unifi.it
\\
\and Irene Sabadini\\ \normalsize Dipartimento di Matematica,
Politecnico di Milano\\ \normalsize Via Bonardi, 9, 20133 Milano,
Italy, irene.sabadini@polimi.it }

\date{August 10, 2008}
\begin{document}
\maketitle
\begin{abstract}
In this paper we show how to construct a regular, non commutative Cauchy kernel for
slice regular quaternionic functions. We prove an (algebraic) representation formula for such functions,
which leads to a new Cauchy formula.
We find the expression of the derivatives of
a regular function in terms of the powers of the Cauchy kernel, and we present several other consequent results.

 \end{abstract}



\section{Introduction}
The interest in creating a theory of quaternionic valued functions of
a quaternionic variable, which would somehow resemble the classical theory of holomorphic
functions of one complex variable, has produced several interesting theories. The classical, best known theory is the one due
to Fueter \cite{fueter 1}, \cite{fueter 2} (see also \cite{csss} for recent developments). In recent papers, adopting an approach used by
Cullen, the authors introduced a new theory of regular functions for quaternionic and
Clifford valued functions,  \cite{cs}, \cite{slicecss}, \cite{gs}, \cite{advances}, \cite{gscliffor}.
For these new classes of functions, that will be called slice regular (resp. slice monogenic),
Cauchy representation formulas were introduced in the afore mentioned papers, by means of kernels that
are not slice regular (resp. not slice monogenic).
Cauchy formulas and Cauchy kernels play a key role in the
study of holomorphic functions and of Fueter regular functions. As in the case of
holomorphic functions, Cauchy formulas are a crucial tool in the construction of
functional calculus for slice
regular and slice monogenic functions. In the quaternionic case, the functional calculus
is associated to
quaternionic operators (see for example \cite{cgssann} and \cite{cgss}), while
in the Clifford algebras case it deals with $n$-tuples of
noncommuting operators
(see \cite{cs}, \cite{functionalcss}).
It is therefore clear that particular attention has to be put in the construction of the Cauchy kernels and the Cauchy formulas
for slice regular functions.
In the present paper, we obtain a new Cauchy formula for slice regular functions
in terms of a regular Cauchy kernel, while
the case of slice monogenic functions is treated in \cite{cs}.

Let $\hh$ be the real associative algebra of quaternions
with respect to the basis $\{1, i,j,k \}$
satisfying the relations
$$
i^2=j^2=k^2=-1,
 ij =-ji =k,\
$$
$$
jk =-kj =i ,
 \  ki =-ik =j .
$$
We will write a quaternion $q$ as $q=x_0+ix_1+jx_2+kx_3$ ($x_i\in
\mathbb{R}$) its conjugate as $\bar q=x_0-ix_1-jx_2-kx_3$, and we
will write $|q|^2=q\ov q$. We will also denote  the real part
$x_0$ of $q$ by ${\rm {\rm Re}}(q)$ and its imaginary part $ix_1 + j x_2 + kx_3$ by ${\rm  Im}(q)$.
The symbol
$\mathbb{S}$ will indicate the $2-$sphere of purely imaginary, unit quaternions,
i.e.
$$
\mathbb{S}=\{ q=ix_1+jx_2+kx_3\ |\  x_1^2+x_2^2+x_3^2=1\}
$$
and we will often use the fact  that any non real quaternion $q$ can be written in a unique way as $q=x+yI$ for $x,y\in \mathbb{R}, y>0$ and $I\in \mathbb{S}$. In particular we will set
$$
I_q=\left\{\begin{array}{l}
\displaystyle\frac{{\rm Im}(q)}{|{\rm Im}(q)|}\quad{\rm if}\ {\rm Im}(q)\not=0\\
{\rm any\ element\ of\ } \mathbb{S}{\rm\ otherwise}\\
\end{array}
\right.
$$

Now, to better explain the problem we deal with, we first of all take the case of the regular functions in the sense of Fueter. Consider the
the (left)  Cauchy--Fueter operator
\begin{equation}\label{Fueter}
{{\pp}\over{\pp\bar q}}={\pp\over{\pp x_0}}+i {\pp\over{\pp x_1}}+j
{\pp\over{\pp x_2}}+k {\pp\over{\pp x_3}}
\end{equation}
If $U$ is an open set of $\mathbb{H}$, then a real differentiable function $f: U \to \mathbb{H}$ is called  (left) Fueter regular if $$
{{\pp}\over{\pp\bar q}}f(q)=0.
$$
for all $q\in U$. By writing the units on the right in (\ref{Fueter}), one obtains the so called right Cauchy-Fueter operator whose kernel gives the right Fueter regular functions. The two theories of left and right Fueter regular functions are completely equivalent.
The function $G(q)$ defined by
\begin{equation}\label{1.1.9}
G(q)={{q^{-1}}\over{|q|^2}}={{\bar q}\over{|q|^4}}
\end{equation}
is called the Cauchy--Fueter kernel, it is a generalization of  the classical Cauchy kernel for holomorphic functions and it is used to find a Cauchy formula. In fact the function $G(q)$ turns out to be left and right Fueter regular on
$\hh\backslash \{0\}$ and we have that: if  $S$ is a four dimensional domain $S\subset U$ whose boundary $\pp S$ is a compact, orientable hypersurface,  and if $q_0$
belongs to the interior of $S$,
then
\begin{equation}\label{1.1.10}
f(q_0)={1\over{2\pi^2}}\int_{\pp S}G(q-q_0 )Dq f(q)
\end{equation}
\par\noindent where, with obvious notations, $Dq$ is the quaternion valued 3-differential form defined by $dx_1\wedge dx_2\wedge dx_3 -idx_0\wedge dx_2\wedge dx_3 -jdx_0\wedge dx_3\wedge dx_1-kdx_0\wedge dx_1\wedge dx_2$.

We will now pass to the case of slice regular functions,  and we will recall the definition  given by the authors in \cite{advances}.

\begin{definition} Let $U\subseteq\hh$ be an open set and let
$f:\ U\to\hh$ be a function. Let $I\in\mathbb{S}$ and let $f_I$ be
the restriction of $f$ to the complex line $L_I := \mathbb{R}+I\mathbb{R}$ passing through $1$ and $I$.
We say that $f$ is a left slice regular (or regular) function if, for every $I\in\mathbb{S}$
$$
\overline{\partial}_I f=\frac{1}{2}\left(\frac{\partial }{\partial x}+I\frac{\partial }{\partial y}\right)f_I(x+Iy)=0,
$$
and we say it is right slice regular (or right regular) if for every $I\in\mathbb{S}$
$$
f \overline{\partial}_I=\frac{1}{2}\left(\frac{\partial }{\partial x}+\frac{\partial }{\partial y}I\right)f_I(x+Iy)=0.
$$
\end{definition}

\noindent We  define the $I-$derivative of $f$ in $q=x+yI$ by
$$\partial_If(x+yI):=\frac{1}{2}\left(\frac{\partial}{\partial x}
-I\frac{\partial}{\partial y}\right)f_I(x+yI)$$
and we are now able to give the following
notion of derivative:
\begin{definition}\label{derivative}
Let $\Omega$ be a domain in $\mathbb{H}$, and let $f:\Omega \to \mathbb{H}$
be a regular function. The slice derivative (in short derivative) of $f$,
$\partial_s f$, is defined as follows:
\begin{displaymath}
\partial_s(f)(x+yI) =
\partial_I(f)(x+yI).
\end{displaymath}
\end{definition}
Notice that the definition of derivative is well posed because it is applied
only to regular functions for which
$$
\frac{\partial}{\partial x}f(x+yI)=
-I\frac{\partial}{\partial y}f(x+yI)\qquad \forall I\in\mathbb{S},
$$
and therefore, analogously to what happens in the complex case,
$$ \partial_s(f)(x+yI) =
\partial_I(f)(x+yI)=\partial_x(f)(x+yI).
$$
Note that if $f$ is a regular function, then its derivative is still regular because
\begin{equation}\label{derivatareg}\overline{\partial}_I(\partial_sf(x+Iy))
=\partial_s(\overline{\partial}_If(x+Iy))=0,\end{equation}
and therefore
$$
\partial^n_sf(x+yI)=\frac{\partial^n
f}{\partial x^n}(x+yI).
$$
\noindent
\par\noindent
For $R>0$, let now $B(0, R)=\{ q\in \mathbb{H} : |q| < R\}$ be the open ball of radius $R$ of $\mathbb{H}$,  let $f:\ B(0,R)\to\hh$ be a slice regular function and let $q=x+yI_q \in B(0,R)$. Then the {\sl Cauchy formula for slice regular quaternionic functions} states that (\cite{advances})
\begin{equation}\label{sliceformul}
f(q)=\frac{1}{2\pi } \int_{\pp\Delta_q(0,r)}
\displaystyle (\zeta -q)^{-1}\, d\zeta_{I_q} \ f(\zeta)
\end{equation}
where $d\zeta_{I_q}=-I_qd\zeta$ and $r>0$ is such that
$
\overline{\Delta_q(0,r)}:=\{x+I_q y\ |\ x^2+y^2\leq r^2\}
$
contains $q$ and is contained in $B(0,R)$.
\par\noindent
It is easy to prove  that the function
$
g(\zeta)=(\zeta -q)^{-1}
$
is not slice regular unless $q\in\mathbb{R}$. It is important at the same time to notice that the function $g$ is effectively used only in the complex plane $L_{I_q}$  which contains the point $q$.
One may wonder if it is possible to consider a variation of the Cauchy formula (\ref{sliceformul}) in which the corresponding kernel is regular and  the path of integration does not depend on the plane $L_{I_q}$ to which the point $q$ belongs.
\par\noindent
The main results in this paper show that both questions have affirmative answers. The key tool used to obtain our results will be
the  function
$$
-(q^2
-2q {\rm Re} [s]+|s|^2)^{-1} (q-\overline s)
$$
which turns out to be the regular inverse $(s-q)^{-*}$ of $R_s(q)=(s-q)$ (see \cite{caterina}).
The function $ (s-q)^{-*}$ is, actually, the unique slice regular extension of
$(s-q)^{-1}$ (in the variable $q$) out of $L_{I_q}$ and
will be called the {\sl noncommutative Cauchy kernel}. In particular, as a
first significant step,  formula (\ref{sliceformul}) can be rewritten in
terms of the noncommutative Cauchy kernel  as
\begin{equation}\label{sliceformul1}
f(q)=\frac{1}{2\pi } \int_{\pp\Delta_q(0,r)}
\displaystyle (\zeta -q)^{-1}\, d\zeta_{I_q} \ f(\zeta)=\frac{1}{2\pi } \int_{\pp\Delta_q(0,r)}
\displaystyle (\zeta -q)^{-*}\, d\zeta_{I_q} \ f(\zeta)
\end{equation}
where, as before, $d\zeta_{I_q}=-I_qd\zeta$ and $r>0$ is such that
$
\overline{\Delta_q(0,r)}:=\{x+I_q y\ |\ x^2+y^2\leq r^2\}
$
contains $q$ and is contained in $B(0,R)$.
The new Cauchy formula that we present in this paper holds naturally for domains - that we will call circular, slice domains - which intersect the real axis and are invariant under the action of purely imaginary rotations in $\mathbb{H}$:

\begin{theorem}
Let $\Omega \subseteq \hh$ be a circular, slice domain  such that $\pp
(\Omega \cap L_I)$ is union of a finite number of rectifiable
Jordan arcs. Let $f$ be a regular function on $\Omega$ and, for
any $I\in \mathbb{S}$,  set  $ds_I=-Ids$. Then for every
$q\in\Omega$ we have:
\begin{equation}\label{integral}
 f(q)=\frac{1}{2 \pi}\int_{\pp (\Omega \cap L_I)} -(q^2-2{\rm Re}(s)q+|s|^2)^{-1}
 (q-\overline{s})ds_I f(s).
\end{equation}
Moreover
the value of the integral depends neither on $\Omega$ nor on the  imaginary unit
$I\in\mathbb{S}$.
\end{theorem}

On circular, slice domains, indeed, the proof of the Cauchy formula is achieved by means of
the following representation formula, which is another crucial result obtained in the paper:

\begin{theorem} Let
$f$ be a regular function on a circular domain $\Omega\subseteq  \mathbb{H}$. Choose any
$J\in \mathbb{S}$.  Then the following equality holds for all $q=x+yI \in \Omega$:
\begin{equation}\label{distribution}
f(x+yI) =\frac{1}{2}\Big[   f(x+yJ)+f(x-yJ)\Big]
+I\frac{1}{2}\Big[ J[f(x-yJ)-f(x+yJ)]\Big].
\end{equation}
\end{theorem}

We present several other consequent results, and we find the expression of the derivatives of
a regular function in terms of the powers of the regular Cauchy kernel.

\section{The Cauchy formula with regular kernel }
\pn
We will recall here the most salient properties of slice regular functions. When no confusion can arise, we
will refer to left slice regular functions simply as regular functions. When it will be needed we will
specify if we are considering left or right regularity.

\begin{remark}{\rm On a domain $U\subseteq\hh$, left regular functions  form a right
$\mathbb{H}$-vector space and
right regular functions  form a left $\mathbb{H}$-vector space.  It is not true, in general, that the product or the composition of two
(left/right) regular functions is (left/right) regular.}
\end{remark}

\noindent One of the key features of this notion of regularity is the fact that
polynomials $\sum_{n=0}^Nq^na_n$ in the quaternion variable $q$, and with quaternionic
coefficients $a_n$, are left regular (while polynomials $\sum_{n=0}^Na_nq^n$ are right regular).
Moreover, any power series $\sum_{n=0}^{+\infty} q^na_n$ (or more in general
$\sum_{n=0}^{+\infty} (q-p_0)^na_n$, $p_0\in\mathbb{R}$)
 is left regular in
its domain of convergence.
Conversely, every regular function on a open ball with center at the origin
can be represented by a power series. Indeed we have, \cite{advances}:
\begin{theorem}\label{serie} If $B=B(0,R)$ is the open ball centered in the origin with radius $R>0$ and
$f:\ B \to\hh$ is a left regular function, then $f$ has a series expansion of the form
$$
f(q)=\sum_{n=0}^{+\infty} q^n\frac{1}{n!}\frac{\pp^nf}{\pp x^n}(0)
$$
converging on $B$. Analogously, if $f$ is right regular it can be expanded as
$$
f(q)=\sum_{n=0}^{+\infty} \frac{1}{n!}\frac{\pp^nf}{\pp x^n}(0) q^n.
$$
In both cases $f$ is infinitely differentiable on $B$
\end{theorem}
\noindent It is straightforward that an analogous statement holds for regular functions in
an open ball centered in any $p_0\in \mathbb{R}$.
Note that, even though the definition of regular
function involves the direction of the unit quaternion $I$, the coefficients of the series
expansion do not depend upon the choice of $I$.

\noindent
A basic result in the theory of regular functions that we will need in the sequel is
the following version of the identity principle (see
\cite{advances}):
\begin{theorem}[Identity Principle] \label{identity principle} Let $f:B(0,R)\to\mathbb{H}$ be a regular function.
Denote by $Z_f=\{q\in B : f(q)=0\}$ the zero set of $f$. If there
exists $I \in \mathbb{S}$ such that $L_I \bigcap Z_f$ has an
accumulation point, then $f\equiv 0$ on $B$.
\end{theorem}

The following definitions can be found in \cite{open}, and are intended to overcome the difficulties arising from the facts that product of regular functions,  and the conjugate of a regular function, are not regular in general.

\begin{definition}
Let $f(q) = \sum_{n=0}^{+\infty}q^{n}a_{n}$ and $g(q) =
\sum_{n=0}^{+\infty}q^{n}b_{n}$ be given quaternionic power series
converging on $B(0,R)$. We define the
\textnormal{regular product} of $f$ and $g$ as the series $f*g(q) =
\sum_{n=0}^{+\infty}q^{n}c_{n}$, where $c_n = \sum_{k=0}^n a_k
b_{n-k}$ for all $n\in\mathbb{N}$.
\end{definition}
\noindent
The series expansion of $f*g$ converges on $B(0,R)$ (see \cite{caterina}), and a similar definition can be given for right regular  functions whose
regular product will have the coefficients on the left. When considering polynomials in
the quaternionic variable $q$ with coefficient on the left (thus right regular functions),
this definition of regular
product coincides with the standard multiplication of polynomials with coefficients in a noncommutative
ring (see, e.g., \cite{lam}).
\begin{definition}\label{symm}
Let $f(q) = \sum_{n=0}^{+\infty}q^{n}a_{n}$ be a given
quaternionic power series with radius of convergence $R$. We
define:
\begin{itemize}
\item[-] the \textnormal{regular conjugate} of
$f$ as the series $f^c(q) = \sum_{n=0}^{+\infty}q^{n}\bar a_{n}$
\item[-] the \textnormal{symmetrization} of $f$ as $f^s = f*f^c$.
\end{itemize}
\end{definition}
\noindent Since $f^s(q) =
\sum_{n=0}^{+\infty}q^{n}r_{n}$, where $r_n = \sum_{k=0}^n a_k \bar
a_{n-k}\in\rr$ for all $n\in\mathbb{N}$, we have that $f^s=f*f^c=f^c*f$ has real coefficients.
An analogous definition, that we will use without stating it explicitly, holds for right
regular power series.

The Cauchy kernel which we will define and study in this paper is inspired by the need
to have a suitable Cauchy formula
to extend the functional calculus for quaternionic operators to functions defined
on more general domains.
Indeed, in the complex case the kernel $(\zeta - z)^{-1}$ is
the sum of the series $\sum_{n \geq 0} z^n \zeta^{-1-n}$, which is obtained by the
standard series development. In the quaternionic case, the same arguments shows
$$
(s-q)^{-1}=((1-qs^{-1})s)^{-1}=s^{-1}(1-qs^{-1})^{-1}
$$
$$
=s^{-1}\sum_{n\geq 0}(qs^{-1})^n=
\sum_{n\geq 0}s^{-1}(qs^{-1})^n.
$$
If we now fix $s=u+vI$ and take $z\in L_I$ with $|z|<|s|$,
the previous expression can be written as
$$
(s-z)^{-1}=\sum_{n \geq 0} z^n s^{-1-n}.
$$
This expression is holomorphic on the disc $\Delta=\Delta (0, |s|)$ in $L_I$ and therefore,
it can be extended uniquely by the Identity Principle
to a regular function on the ball $B(0, |s|)$ in $\mathbb{H}$
$$
\sum_{n \geq 0} q^n s^{-1-n}.
$$
This explains the meaning of the following definition (see
\cite{cgssann}, \cite{cgss}):
\begin{definition}\label{Cauchykernel}
Let $q$, $s\in \mathbb{H}$ such that $sq\not=qs$.
We will call non commutative Cauchy kernel series (shortly Cauchy kernel series)
the series expansion
$$
S^{-1}(s,q):=\sum_{n\geq 0}q^ns^{-1-n},
$$
for $|q|<|s|$.
\end{definition}
\begin{proposition}
The Cauchy kernel series is left regular in $q$ and right regular in
$s$, respectively, for $|q|<|s|$.
\end{proposition}
\begin{proof}
It is an immediate consequence of Theorem \ref{serie}.
\end{proof}
We will now show two possible approaches to the construction of a regular Cauchy kernel
function. The first one is direct: we compute the
Cauchy kernel function and then we show that it is regular.
The second one makes use of the so called regularization process
introduced in \cite{caterina}.

\noindent Let us start by the direct approach
introduced, in a different setting, in \cite{cgssann},
\cite{cgss}. In these two papers we proved the following:
\begin{theorem}\label{equazione per S}
Let $q$
and $s$ be two quaternions such that $qs\not=sq$ and consider
$$
S^{-1}(s,q):=\sum_{n\geq 0} q^n s^{-1-n}.
$$
Then the  inverse $S(s,q)$ of the quaternion $S^{-1}(s,q)$ is the nontrivial solution to the equation
\begin{equation}\label{quadratica}
S^2+Sq-sS=0.
\end{equation}
\end{theorem}
\begin{remark} {\rm An algebraic equation with quaternionic coefficients can be
suitably rewritten with the coefficients on one side. When algebraic equations are written
with coefficients on one side, they may have isolated zeroes
or 2-dimensional spheres of solutions. In particular, a degree two
equation has either two isolated zeroes or a 2-sphere of zeroes
(see e.g. \cite{lam}). Since $S=0$ is a trivial solution of
(\ref{quadratica}), also the second solution must be isolated.}
\end{remark}
\begin{remark}{\rm
Note that $R(s,q)=s-q$ is a solution of equation
(\ref{quadratica}) if and only if $sq=qs$.}
\end{remark}
\noindent
We have the following result:
\begin{theorem} \label{thinvquat}
Let $q$, $s\in \mathbb{H}$ be such that
$qs\not=sq$. Then the non trivial solution of
\begin{equation}\label{Sequa}
S^2+Sq-sS=0
\end{equation}
 is given by
$$
S(s,q)=(q-\overline s)^{-1}s(q-\overline s)-q
$$
\begin{equation}\label{Ssolution}
=-(q-\overline s)^{-1}(q^2
-2q {\rm Re} (s)+|s|^2).
\end{equation}
\end{theorem}
\begin{proof}
The result has been proved in \cite{cgss} by directly verifying that $S(s,q)$ is a solution.
Here we show how to find the solution using the techniques developed in
\cite{ps}, \cite{serodio zeri}. We remark that
it is possible to find the same result also using the so-called
Niven's  algorithm method which however involves longer computations.
We transform the equation $S^2+Sq-sS=0$ into another one having
coefficients on the left.
Set
$$
S:=W-q
$$
and replace it in equation (\ref{Sequa}) to get
$$
(W-q)(W-q)+(W-q)q-s(W-q)=0,
$$
so the equation becomes
$$
W^2-(s+q)W+sq=(W-s)*(W-q)=0,
$$
where $*$ denotes the left regularized product.
One root is $W=q$  (see \cite{serodio zeri}), while the second root is $W=(q-\bar s)^{-1}s(q-\bar s)$,
thus $S=(q-\bar s)^{-1}s(q-\bar s)-s$. By grouping $(q-\bar s)^{-1}$ on the left we obtain
(\ref{Ssolution}).
\end{proof}
\begin{definition}
The function defined by
\begin{equation}\label{esse-1}
S^{-1}(s,q)=-(q^2-2q{\rm Re}(s)+|s|^2)^{-1}(q-\bar s).
\end{equation}
will be called Cauchy kernel function.
\end{definition}
Note that we are using the same symbol $S^{-1}$ to denote both the Cauchy kernel series and the Cauchy
kernel function. Indeed they coincide where they are both defined by virtue of their regularity
(see Proposition \ref{regularboth}) and in view of the Identity Principle.
\begin{proposition}
For any $q,s\in\mathbb{H}$ such that $q\not= \bar{s}$ the following identity holds:
\begin{equation}
(q-\overline s)^{-1}s(q-\overline s)-q=(s-\bar q)q(s-\bar q)^{-1}+s,
\end{equation}
or, equivalently,
\begin{equation}\label{second}
-(q-\overline s)^{-1}(q^2
-2q {\rm Re} (s)+|s|^2)=(s^2-{\rm Re}(q)s+|q|^2)(s-\bar q)^{-1}.
\end{equation}
\end{proposition}
\begin{proof}
One may prove the identities by direct computations but we prefer to follow here a shorter approach.
Let us solve equation (\ref{Sequa}) by transforming it into an
equation with right coefficients by setting
$$
S:=W+s
$$
and replacing it in the equation. We get
$$
(W+s)(W+s)+(W+s)q-s(W+s)=
W^2+W(s+q)+sq=0.
$$
This equation can be split as
$(W+s)*(W+q)=0$, where $*$ denotes the (right) regular product.
It is immediate, see \cite{caterina}, that one root is $W=-s$ while the second is $W=(-\bar q +s)q(-\bar q+s)^{-1}$.
These two roots correspond to $S=0$ and $S=(s-\bar q)q(s-\bar q)^{-1}+s$ which coincides
with (\ref{Ssolution})
when written in the form $S=(s^2-{\rm Re}(q)s+s^2)(s-\bar q)^{-1}$. \end{proof}

\noindent As announced, let us now follow the second approach to find
$S^{-1}(s,q)$. When $q$, $s$ commute, the sum of the Cauchy kernel series  is
the function $R^{-1}(q)=(s-q)^{-1}$.
Since this last function is, in general,
not regular, we
construct its regular extension. A crucial point is that the regular
extension of $R^{-1}$ coincides with the sum of the Cauchy kernel series.
\begin{proposition}
Let $f(q)$ be a regular function in $B(0,R)$. Its inverse with respect to the regular  product is the function
$$
f^{-*}(q)=(f^s(q))^{-1}f^c(q).
$$
The function $f^{-*}$ is regular on $B(0,R)\setminus \{ q\in \hh\ : \ f^s(q)=0\}$.
\end{proposition}
\begin{proof}
An easy computation shows that
$$
(f^{-*}*f)(q)=(f^s(q))^{-1}(f^c*f)(q)=1.
$$
Moreover, the series expansion $f^s(q) = \sum_{n=0}^{+\infty}q^{n}r_{n}$ has real coefficients and hence
$(f^s(q))^{-1}$ is regular; indeed, for $I\in\mathbb{S}$, if we set $z=x+yI$ we have:
$$
\frac{\pp}{\pp x}(f^s(x+yI))^{-1}=-(f^s(x+yI))^{-2}\frac{\pp}{\pp x}f^s(x+yI)
$$
$$
\frac{\pp}{\pp y}(f^s(x+yI))^{-1}=-(f^s(x+yI))^{-2}\frac{\pp}{\pp y}f^s(x+yI)I=
-I(f^s(x+yI))^{-2}\frac{\pp}{\pp y}f^s(x+yI).
$$
It is immediate now to verify that
$\left(\frac{\partial }{\partial x}+I\frac{\partial }{\partial y}\right)(f^s(x+yI))=0$
for all $I\in\mathbb{S}$. The regularity of $f^{-*}$ follows by the regularity of $f^c$
and  by the fact that $(f^s)^{-1}*f^c=(f^s)^{-1}f^c$ since $(f^s)^{-1}$ has real coefficients.
\end{proof}
\begin{remark}\label{regularinverse}{\rm We point out that $(f^s)^{-1}*f^c(q)=(f^s)^{-1}f^c(q)=f^c*(f^s)^{-1}(q)$.
If we construct a regular inverse with respect to the left regularized product, we get an analogous formula
where the symbol $*$ denotes the left regularized product:
$(f^s)^{-1}*f^c(q)=f^c*(f^s)^{-1}(q)=f^c(f^s)^{-1}(q)$.
Note that the only difference is the position of $f^c$ when using the standard product. }
\end{remark}
\noindent By applying the previous proposition to $R(s,q)=s-q$ we obtain:
\begin{proposition}
The inverse with respect to the regular product of the function $R(s,q)=R_s(q)=s-q$ is
$S^{-1}(s,q)$.
\end{proposition}
\begin{proof} We have that $R^c_s(q)=\bar s-q$ and $R^s_s(q)=(q^2-2{\rm Re}(s)q+|s|^2)$.
\end{proof}
\noindent We also have:
\begin{proposition}\label{regularboth}
The function $S^{-1}(s,q)$ is left regular in the variable $q$ and
right regular in the variable $s$ in its domain of definition.
\end{proposition}
\begin{proof}
The regularity in $q$ follows by construction. The right regularity in $s$ follows by
direct computations.
\end{proof}
\begin{remark}{\rm The right regular inverse of the function $R(s,q)=R_q(s)$ is the function $(s-\bar q)(s^2-{\rm Re}(q)s+|q|^2)^{-1}$ (see remark \ref{regularinverse}) which, by  identity (\ref{second}),  coincides with $S^{-1}(s,q)$.
 }\end{remark}
\begin{definition}
Given a quaternion $q=x+I_p y$ $(x,y \in \mathbb{R}$ and $I_q\in \mathbb{S})$,
we will denote  by $\Sigma_q=x+\mathbb{S}y$ the $2-$sphere of its (generalized) conjugates.
\end{definition}
\noindent Note that the 2-sphere $\Sigma_q$ is equivalently defined
by the conditions $|q|=x^2+y^2$ and ${\rm Re}(q)=x$.
\begin{proposition}
Let $q\in\hh\backslash\rr$. The singularities of the function $S^{-1}(s,q)=S^{-1}_q(s)=
-(q^2-{\rm Re}(s)q+|s|^2)^{-1}(q-\bar s)$
lie on the 2-sphere $S_q$.
More precisely: on the plane $L_I$,
$I\not=I_q$ the restriction of the function $S^{-1}(s,q)$ to $L_{I}$ has the two
singularities ${\rm Re}(q)\pm I|{\rm Im}(q)|$; while on the plane $L_{I_q}$
the restriction of the function $S^{-1}(s,q)$ to $L_{I_q}$
has the only singularity $q$.
When $q\in\rr$, then $S^{-1}_q(s)=(q-s)^{-1}$ and the only singularity is $q$.
\end{proposition}
\begin{proof}
Suppose $q\in\hh\backslash\rr$.
The singularities of $S^{-1}_q(s)$ correspond to the roots of $|s|^2-2{\rm Re}(s)q+q^2=0$.
This equation can be written by splitting real and imaginary parts as
$|s|^2-2{\rm Re}(s){\rm Re}(q)+{\rm Re}(q)^2-|{\rm Im}(q)|^2=0$, $({\rm Re}(s)-{\rm Re}(q))
{\rm Im}(q)=0$. This implies
${\rm Re}(s)={\rm Re}(q)$ and so
$|s|=|q|$ i.e. the singularities consist of the whole
2-sphere $\Sigma_q$. Consider the plane $L_I$, $I\not=I_q$: it intersects the 2-sphere
$\Sigma_q$ in ${\rm Re}(q)\pm I|{\rm Im}(q)|$. When $I=I_q$, or $q$ is real, then $q$ and $s$ commute, so $S^{-1}_q(s)=-(q-s)^{-1}
(q-\bar s)^{-1} (q-\bar s)=(q-s)^{-1}$ and the statement follows.
\end{proof}

\subsection{The Cauchy integral formula}
In this section we will present a new Cauchy formula  (\ref{sliceformul}),
which holds for functions defined on a larger class of domains.
\begin{lemma}\label{lemma1}
Let $f$, $g$ be quaternionic valued, continuously (real) differentiable functions on an open
set $U_I$ of the plane $L_I$.  Then for every open $W\subset
U_I$ whose boundary consists of a finite number of piecewise smooth,  closed curves  we have
$$
\int_{\partial W} g ds_I f=2\int_W ((g\overline{\pp}_I)f+g(\overline{\pp}_I f)) d\sigma ,
$$
where $s=x+Iy$ is the variable on $L_I$, $ds_I=-I ds$ and
$d\sigma=dx\wedge dy$.
\end{lemma}
\begin{proof}
Let us choose an imaginary unit $J$ orthogonal to $I$ and let us
consider $\hh$ as the algebra generated by $I,J$. Then it is
possible to write $f(s)=f_0(s)+f_1(s)J$, $g(s)=g_0(s)+Jg_1(s)$ for
suitable $L_I-$valued functions $f_i(s),g_i(s)$, $i=0,1$. We now use the usual
Stokes' theorem to write
$$
\int_{\partial W} g ds_I f=\int_{\partial W} (g_0(s)+Jg_1(s)) ds_I
(f_0(s)+f_1(s)J)
$$
$$
= \int_{\partial W} g_0f_0 ds_I +g_0f_1
ds_I J +Jg_1f_0 ds_I + J g_1f_1 ds_I J
$$
$$
=\int_W \pp_x(g_0f_0) d\sigma+\pp_y(g_0f_0) I d\sigma +\pp_x(g_0f_1)
d\sigma J+\pp_y(g_0f_1)
I d \sigma J +
$$
$$
+J\pp_x(g_1f_0)d \sigma+J\pp_y(g_1f_0) I
d\sigma  + J\pp_x( g_1f_1)d\sigma J+ J\pp_y( g_1f_1)Id\sigma J.
$$
By direct computations we get
$$
\pp_x(g_0f_0)+\pp_y(g_0f_0) I +J\pp_x(g_1f_0)+J\pp_y(g_1f_0) I
$$
$$
=(\pp_x(g_0)+\pp_y(g_0)I)f_0 +J(\pp_x(g_1)+\pp_y(g_1)I)f_0
$$
$$
+ g_0(\pp_x(f_0)+\pp_y(f_0)I) +Jg_1(\pp_x(f_0)+\pp_y(f_0)I)
$$
$$
=(g_0\overline{\pp}_I +Jg_1\overline{\pp}_I )f_0
$$
$$
+(g_0+Jg_1)(\overline{\pp}_If_0)=2(g\overline{\pp}_I)f_0+2g(\overline{\pp}_If_0)
$$
and analogously
$$
\pp_x(g_0f_1)J +\pp_y(g_0f_1)IJ+ J\pp_x( g_1f_1)J + J\pp_y( g_1f_1)IJ = 2(g\overline{\pp}_I)f_1J+2g(\overline{\pp}_If_1)J.
$$
Therefore we conclude
$$
\int_{\partial W} g ds_I f =2\int_W (g\overline{\pp}_I)f_0d\sigma+g(\overline{\pp}_If_0)d\sigma+(g\overline{\pp}_I)f_1Jd\sigma+g(\overline{\pp}_If_1)Jd\sigma
$$
$$
= 2\int_W ((g\overline{\pp}_I)f+g(\overline{\pp}_I f)) d\sigma .
$$
\end{proof}
An immediate consequence of the Lemma is the following:
\begin{corol}\label{int_nullo}
Let $f$ and $g$ be  a left regular and a right regular function, respectively, on an open set $U\in \mathbb{H}$.
For any $I\in\mathbb{S}$ and
every open $W\subset
U_I$ whose boundary consists of a finite number of piecewise smooth,  closed curves,  we have:

$$
\int_{\partial W} g ds_I f=0.
$$
\end{corol}

We will now identify a class of domains that naturally qualify as domains of definition of regular functions.

\begin{definition}
Let $\Omega \subseteq \hh$ be a domain in $\mathbb{H}$. We say that $\Omega$ is a
\textnormal{slice domain (s-domain for short)} if $\Omega \cap \mathbb{R}$ is non empty and if $L_I\cap \Omega$ is a domain in $L_I$ for all $I \in \mathbb{S}$.
\end{definition}

A regular function defined on a s-domain has a quaternionic series expansion which converges in a small open ball centered at  a real point of the domain (see theorem \ref{serie}), and hence  it is very easy to prove
that the identity principle stated in theorem \ref{identity principle} holds for regular functions defined on s-domains. For our purpose here we will also make use of the following

\begin{definition}
Let $\Omega \subseteq \hh$. We say that $\Omega$ is
\textnormal{circular} if, for all $x+yI \in \Omega$, the whole
2-sphere $x+y\mathbb{S}$ is contained in $\Omega$.
\end{definition}

We will focus our attention on the study of regular functions defined on circular s-domains.
In fact these functions and domains turn out to be the natural setting of validity of the
Cauchy formula that we are going to present. To construct this formula, we need now to extend
to a wider class of regular functions the following value distribution property of regular
quaternionic power series proved in \cite{caterina}:

\begin{theorem}\label{values}
Let $f : B=B(0,R) \to \hh$ defined by $f(q)=\sum_{n\ge 0} q^na_n$ be a regular function. For all $x,y \in
\rr$ such that $x+y\mathbb{S} \subseteq B$ there exist  $b,c \in
\hh$ such that
\begin{equation}
f(x+yI)=b+I c
\end{equation}
for all $I \in \mathbb{S}$.
\end{theorem}

\noindent We will in fact prove here the following extension of
theorem \ref{values} (see \cite{cs} for the case of slice monogenic functions):

\begin{theorem}[Representation Formula]\label{formula} Let
$f$ be an regular function on a circular domain $\Omega\subseteq  \mathbb{H}$. Choose any
$J\in \mathbb{S}$.  Then the following equality holds for all $q=x+yI \in \Omega$:
\begin{equation}\label{distribution}
f(x+yI) =\frac{1}{2}\Big[   f(x+yJ)+f(x-yJ)\Big]
+I\frac{1}{2}\Big[ J[f(x-yJ)-f(x+yJ)]\Big].
\end{equation}
Moreover, for all $x, y \in \mathbb{R}$ such that $x+y\mathbb{S} \subseteq \Omega$, there exist $b, c \in \mathbb{H}$ such that for all $K \in \mathbb{S}$ we have
\begin{equation}\label{cappa}
\frac{1}{2}\Big[   f(x+yK)+f(x-yK)\Big]=b \quad \quad {\sl and}
\quad \quad \frac{1}{2}\Big[ K[f(x-yK)-f(x+yK)]\Big]=c.
\end{equation}
\end{theorem}
\begin{proof} If ${\rm Im}(q)=0$ is real, the proof is immediate. Otherwise
let us define the function $\psi:\Omega \to \mathbb{H}$ as follows
$$
\psi(q)=\frac{1}{2}\Big[   f({\rm Re}(q) +|{\rm Im}(q)|J)+f({\rm Re}(q) - |{\rm
Im}(q)|J)
$$
$$
+\frac{{\rm Im}(q)}{|{\rm Im}(q)|}J[f({\rm Re}(q) - |{\rm Im}(q)|J)-f({\rm Re}(q) + |{\rm
Im}(q)|J)]\Big].
$$
Using the fact that $q=x+yI$, $x,y\in\mathbb{R}$, $y\geq 0$ and
$I=\displaystyle\frac{{\rm Im}(q)}{|{\rm Im}(q)|}$ we obtain
$$
\psi(x+yI)=\frac{1}{2}\Big[   f(x+yJ)+f(x-yJ) +I
J[f(x-yJ)-f(x+yJ)]\Big].
$$
Observe that on $L_J$ (i.e. for $I=J$) we have
$$
\psi_J(q)=\psi(x+yJ)=f(x+yJ)=f_J(q).
$$
Therefore if we prove that $\psi$ is regular on $\Omega$,
the first part of the assertion will follow from the Identity Principle
for regular functions.
Indeed, since $f$ is regular on $\Omega$, for any $I\in \mathbb{S}$ we have,
on $\Omega\cap L_{I}$
$$
\frac{\partial}{\partial x}2\psi(x+yI)
=\frac{\partial}{\partial x} \Big[f(x+yJ)+f(x-yJ) +I J[f(x-yJ)-f(x+yJ)]\Big]
$$
$$
=\frac{\partial}{\partial x} f(x+yJ)+\frac{\partial}{\partial x} f(x-yJ) +IJ[\frac{\partial}{\partial x} f(x-yJ)-\frac{\partial}{\partial x} f(x+yJ)]
$$
$$
=-J\frac{\partial}{\partial y} f(x+yJ)+J\frac{\partial}{\partial y} f(x-yJ) +I J[J\frac{\partial}{\partial y} f(x-yJ)+J\frac{\partial}{\partial y} f(x+yJ)]
$$
$$
=-J\frac{\partial}{\partial y} f(x+yJ)+J\frac{\partial}{\partial y} f(x-yJ) -I [\frac{\partial}{\partial y} f(x-yJ)+\frac{\partial}{\partial y} f(x+yJ)]
$$
$$
=-I\frac{\partial}{\partial y} \Big[f(x+yJ)+f(x-yJ) +I J[f(x-yJ)-f(x+yJ)]\Big]=-I\frac{\partial}{\partial y}2\psi(x+yI)
$$
i.e.
\begin{equation}\label{reg}
\frac{1}{2}(\frac{\partial}{\partial x}+I\frac{\partial}{\partial
y})\psi(x+yI)=0.
\end{equation}
To prove (\ref{cappa}) we take any $K\in \mathbb{S}$ and use
equation (\ref{distribution}) to show that
$$
\frac{1}{2}\Big[   f(x+yK)+f(x-yK)\Big]
$$
$$
=\frac{1}{2}\Big\{\frac{1}{2}\Big[   f(x+yJ)+f(x-yJ)\Big] +K\frac{1}{2}\Big[ J[f(x-yJ)-f(x+yJ)]\Big]
$$
$$
+\frac{1}{2}\Big[   f(x+yJ)+f(x-yJ)\Big] -K\frac{1}{2}\Big[ J[f(x-yJ)-f(x+yJ)]\Big]\Big\}
$$
$$
=\frac{1}{2}\Big[   f(x+yJ)+f(x-yJ)\Big]
$$
and that
$$
\frac{1}{2}\Big[ K[f(x-yK)-f(x+yK)]\Big]
$$
$$
=\frac{1}{2}K\Big\{\frac{1}{2}\Big[   f(x+yJ)+f(x-yJ)\Big] -K\frac{1}{2}\Big[ J[f(x-yJ)-f(x+yJ)]\Big]
$$
$$
-\frac{1}{2}\Big[   f(x+yJ)+f(x-yJ)\Big] -K\frac{1}{2}\Big[ J[f(x-yJ)-f(x+yJ)]\Big]\Big\}
$$
$$
=\frac{1}{2}K\Big[-K\Big[ J[f(x-yJ)-f(x+yJ)]\Big]
$$
$$
=\frac{1}{2}\Big[ J[f(x-yJ)-f(x+yJ)]\Big].
$$
With these two last equalities, the proof is completed.
\end{proof}
The Representation Formula plays a key role in the study of the
theory of regular functions on circular s-domains and we provide
here some of its important consequences. Further developments will
be the subject of a forthcoming paper.
\begin{corol}\label{hartogs}
An regular function $f : \Omega\to\mathbb{H}$ on a circular
s-domain is infinitely differentiable on $\Omega$.
\end{corol}
\begin{proof}
 The differentiability of $f$
on  the real axis follows from Theorem \ref{serie} since for
any point of the real axis there is a ball in which the function
$f$ can be expressed in power series.  To prove differentiability
outside the real axis consider formula (\ref{distribution}) in
terms of $q=x_0+ix_1+jx_2+kx_3$, namely
$$
f(q)=\frac{1}{2}\Big[   f({\rm Re}(q) +|{\rm Im}(q)|J)+f({\rm
Re}(q) - |{\rm Im}(q)|J)
$$
$$
+\frac{{\rm Im}(q)}{|{\rm Im}(q)|}J[f({\rm Re}(q) - |{\rm
Im}(q)|J)-f({\rm Re}(q) + |{\rm Im}(q)|J)]\Big].
$$
Notice that the function $f$ is regular and  hence infinitely differentiable on $L_J$.
It is therefore obvious that $f$ can be obtained as a composition of the functions
$f_J$,
${\rm Re}(q)$, ${\rm Im}(q)$ and $|{\rm Im}(q)|$ which are all infinitely differentiable
outside the real axis with respect to the variables
$x_\ell$, $\ell=0,\ldots ,3$. This concludes the proof. \end{proof}

\noindent The following results are geometry-flavored consequences of  the above theorem:

\begin{corol}
\label{values general}
Let $\Omega\in \mathbb{H}$  be a circular s-domain and let $f : \Omega \to \hh$
be a regular function.
\begin{enumerate}
\item For all $x,y \in
\rr$ such that $x+yI\in \Omega$ there exist $b, c\in \mathbb{H}$ such that
\begin{equation}\label{b e c}
f(x+yI)=b+Ic
\end{equation}
for all $I\in \mathbb{S}$. In particular the image $f(x+y\mathbb{S})$ of the $2-$sphere $x+y\mathbb{S}$ is the $2-$sphere $b+\mathbb{S}c$.
\item If $f(x+yJ)=f(x+yK)$ for $I\neq K$ in $\mathbb{S}$, then $f$ is constant on $x+y\mathbb{S}$;
In particular, if $f(x+yJ)=f(x+yK)=0$ for $I\neq K$ in $\mathbb{S}$, then $f$ vanishes on the entire $2-$sphere $x+y\mathbb{S}$.
\end{enumerate}
\end{corol}
\begin{proof}
The proof of {\sl1.} is a direct application of theorem \ref{formula}. To prove {\sl 2.} notice that $f(x+yJ)=f(x+yK)$ for $I\neq K$ implies $c=0$ in (\ref{b e c}), and hence the assertion follows.
\end{proof}
We can now prove the new version of the Cauchy formula, which makes use
of the class of circular s-domains  naturally containing all the
singularities of the regular kernel (for the case of slice monogenic functions see \cite{cs}).

\begin{theorem}
Let $\Omega \subseteq \hh$ be a circular s-domain  such that $\pp
(\Omega \cap L_I)$ is union of a finite number of rectifiable
Jordan arcs. Let $f$ be a regular function on $\Omega$ and, for
any $I\in \mathbb{S}$,  set  $ds_I=-Ids$. Then for every
$q\in\Omega$ we have:
\begin{equation}\label{integral}
 f(q)=\frac{1}{2 \pi}\int_{\pp (\Omega \cap L_I)} -(q^2-2{\rm Re}(s)q+|s|^2)^{-1}
 (q-\overline{s})ds_I f(s).
\end{equation}
Moreover
the value of the integral depends neither on $\Omega$ nor on the  imaginary unit
$I\in\mathbb{S}$.
\end{theorem}
\begin{proof}
First of all, the integral does not depend on the open set $\Omega$.
This follows from the right regularity of $S^{-1}(s,q)$ with respect to  the variable $s$
(see equality (\ref{second})
and Corollary \ref{int_nullo}).

Let us show that the integral does not depend on the choice of the
imaginary unit $I\in\mathbb{S}$.
If we denote $q=x+yI_q\in\Omega$, then
 the
set of the zeroes of the function $q^2-2{\rm Re}(s)q+|s|^2=0$ consists of
a real point (counted twice) or a 2-sphere. If the zeroes are not
real, on any complex plane $L_I$ we find the two zeroes
$s_{1,2}=x\pm y I$. When the singularity is a real
number, the integral reduces to the classical Cauchy integral formula for
holomorphic maps. Thus we consider the case of nonreal zeroes and we
calculate the residues about the points $s_{1}$ e $s_{2}$. Let us
start with $s_{1}=x+yI$ setting the positions
$$
s=x +y I+\varepsilon e^{I\theta},
$$
$$
{\rm Re}(s)=x+\varepsilon \cos\theta,
$$
$$
\overline{s}=x -y I+\varepsilon e^{-I\theta},
$$
$$
ds_I=-I [\varepsilon I e^{I\theta}]  d\theta=\varepsilon
e^{I\theta}  d\theta,
$$
$$
|s|^2=x^2+2x\varepsilon \cos\theta+\varepsilon^2+y^2+2y\varepsilon \sin\theta .
$$
We now compute the integral which appears at the right hand side of (\ref{integral})
along the circle with center at $s_1$ and radius
$\varepsilon >0$ on the plane $L_I$:
$$
2\pi I_{1}^\varepsilon=\int_{0}^{2\pi}-(-2q\varepsilon \cos\theta+
2x \varepsilon\cos\theta+\varepsilon^2+2 y \varepsilon
\sin\theta )^{-1}
(q-[x -y I+\varepsilon
e^{-I\theta}])\varepsilon  e^{I\theta}  d\theta f(x +y I+\varepsilon e^{I\theta})
$$
$$
=\int_{0}^{2\pi}-(-2q \cos\theta+
2x \cos\theta+\varepsilon+2 y
\sin\theta )^{-1}
(q-[x -y I+\varepsilon
e^{-I\theta}])  e^{I\theta}  d\theta f(x +y I+\varepsilon e^{I\theta}).
$$
For $\varepsilon\to 0$ we get an expression $I_1^0$ for the residue at $s_1$
$$
2\pi I_{1}^0=\int_{0}^{2\pi}(2q \cos\theta- 2x
\cos\theta-2y \sin\theta )^{-1} (y I_q +y I) e^{I\theta}  d\theta f(x +y I)
$$
$$
=\frac{1}{2}\int_{0}^{2\pi}(y\cos\theta I_q -y \sin\theta )^{-1} (y I_q +y I) e^{I\theta}  d\theta
f(x +y I)
$$
$$
=-\frac{1}{2y^2}\int_{0}^{2\pi}(y
\cos\theta I_q+y \sin\theta ) (y I_q +y I)
[\cos \theta+I\sin\theta]  d\theta f(x +y I)
$$
$$
=-\frac{1}{2y^2} \int_{0}^{2\pi} [(y I_q)^2
\cos\theta+y^2\sin\theta  I_q +y^2\cos\theta I_q I  +  y^2 \sin\theta I ]
[\cos \theta+I\sin\theta]  d\theta f(x +y I)
$$
$$
=-\frac{1}{2} \int_{0}^{2\pi} [-
\cos\theta+\sin\theta  I_q +\cos\theta I_q I  +   \sin\theta I ]
[\cos \theta+I\sin\theta]  d\theta f(x +y I)
$$
$$
=-\frac{1}{2} \int_{0}^{2\pi} [-\cos^2\theta-\cos\theta \sin\theta  I+\cos\theta \sin\theta  I_q
+\sin^2\theta I_q I +\cos^2\theta I_q I
$$
$$
-\cos\theta \sin\theta  I_q+\cos\theta \sin\theta  I
-   \sin^2\theta  ]
  d\theta f(x +y I)
$$
$$
=-\frac{1}{2} \int_{0}^{2\pi}
\Big[-1 +   I_q I
 \Big]d\theta f(x +y I)
$$
$$
 ={\pi}
\Big[1 -   I_q I \Big] f(x +y I).
$$
So we get the first
residue
$$
I_{1}^0 = \frac{1}{2}\Big[1 -  I_q I\Big]f(x +y I).
$$
With analogous calculation we prove that the residue about $s_2$
is
$$
I_{2}^0 = \frac{1}{2}\Big[1 +  I_q I\Big]f(x -y I).
$$
By the classical  residues theorem used in the complex
plane $L_I$, we have:
$$
\frac{1}{2 \pi}\int_{\pp (U\cap L_I)} S^{-1}(s,q)ds_I f(s)
=I_{1}^0+I_{2}^0.
$$
Since
$$
I_{1}^0+I_{2}^0=\frac{1}{2}\Big[1 -  I_q I\Big]f(x +y I) +\frac{1}{2}\Big[1 +  I_q I\Big]f(x -y I)
$$
$$
=\frac{1}{2}\Big[   f(x+yI)+f(x-yI)\Big] +I_q\frac{1}{2}\Big[ I[f(x-yI)-f(x+yI)]\Big],
$$
the statement now follows from Lemma \ref{formula}.
\end{proof}

\begin{corol}
Let $I\in\mathbb{S}$ and let $\Omega_I$ be a domain in $L_I$, symmetric with respect to the real
axis and whose boundary is a finite union of closed Jordan arcs. Suppose that $\Omega_I\cap \mathbb{R}$ is nonempty. Let $J\in\mathbb{S}$ be orthogonal to $I$,
 let $F,G:\ \Omega_I \to L_I$ be holomorphic functions and let $f(x+yI)=F(x+yI)+G(x+yI)J$.
 If
 $$\Omega=\bigcup_{x+yI\in\Omega_I} (x+y\mathbb{S})$$ then the function defined by
$$
 \tilde f(q)=\frac{1}{2 \pi}\int_{\pp \Omega_I} -(q^2-2Re(s)q+|s|^2)^{-1}
 (q-\overline{s})ds_I f(s)
 $$
is the regular extension of $f$ to $\Omega$.
\end{corol}

\begin{proposition}[Derivatives using the regular Cauchy kernel]
Let $U\subset\mathbb{H}$ be a circular s-domain.
Suppose
$\partial (U\cap L_I)$ is a finite union of
rectifiable Jordan curves  for every $I\in\mathbb{S}$.
Let $f$ be an regular function on $U$ and set  $ds_I=ds/ I$.
 Let $q$,
$s$.
Then
 $$
\partial^n_{x} f(q)
=\frac{n!}{2 \pi}
\int_{\partial (U\cap L_I)}  (q^2-2s_0q+|s|^2)^{-n-1} (q-\overline{s})^{*(n+1)} ds_I f(s)
$$
\begin{equation}\label{quattordici}
=\frac{n!}{2 \pi}
\int_{\partial (U\cap L_I)}  [S^{-1}(s,q)(q-\overline{s})^{-1}]^{n+1}
(q-\overline{s})^{*(n+1)} ds_I f(s)
\end{equation}
where
\begin{equation}\label{stellina}
(q-\overline{s})^{*n}=\sum_{k=0}^{n}\frac{n!}{(n-k)!k!} q^{n-k}\overline{s}^k,
\end{equation}
is the $n$-th power with respect to the $*$-product. Moreover,
the integral does not depend on $U$ and on  the imaginary
unit $I\in \mathbb{S}$.
\end{proposition}
\begin{proof}
First of all, we recall that the derivative
coincides, for regular functions,
with the partial derivative with respect to the scalar coordinate $x={\rm Re}(q)$.
Therefore, we can identify
$\partial^n_{x} f(q)=\partial^n_{s} f(q)$.
To compute $\partial^n_{x} f(q)$,it is enough to compute the derivative of
the integrand, since $f$ and its derivatives with respect to $x$ are continuous functions on
$\partial (U\cap L_I)$.
Thus
we get
 $$
\partial^n_{x} f(q)=\frac{1}{2 \pi}
\int_{\partial (U\cap L_I)}  \partial^n_{s}[S^{-1}(s,q)] ds_I f(s).
$$
To prove the statement, it is sufficient to compute
$\partial^n_{x}[S^{-1}(s,q)]$. We proceed  by recurrence.
Consider the derivative $\partial_{x}S^{-1}(s,q)$:
$$
\partial_{x}S^{-1}(s,q)=-(q^2-2{\rm Re}(s)q+|s|^2)^{-2} (2q-2{\rm Re}(s))(q-\bar s)-
(q^2-2{\rm Re}(s)q+|s|^2)^{-1}
$$
$$
=(q^2-2{\rm Re}(s)q+|s|^2)^{-2} [2q^2-2q\overline{s}
-2{\rm Re}(s)q+2{\rm Re}(s)\overline{s}-q^2+2{\rm Re}(s)q-|s|^2]
$$
$$
=(q^2-2{\rm Re}(s)q+|s|^2)^{-2} [q^2-2q\overline{s} +\overline{s}^2]=(q^2-2{\rm Re}(s)q+|s|^2)^{-2}
 (x-\overline{s})^{*2}.
$$
We now assume
$$
\partial^n_{x}S^{-1}(s,q)=(-1)^{n+1} n!  (q^2-2{\rm Re}(s)q+|s|^2)^{-(n+1)}
(q-\overline{s})^{*(n+1)},
$$
and we compute $\partial^{n+1}_{x}S^{-1}(s,q)$. We have:
$$
\partial^{n+1}_{x}S^{-1}(s,q)=\partial_{x}[(-1)^{n+1} n!
 (q^2-2{\rm Re}(s)q+|s|^2)^{-(n+1)} (q-\overline{s})^{*(n+1)}]
$$
$$
=(-1)^{n+2} (n+1)!  (q^2-2{\rm Re}(s)q+|s|^2)^{-(n+2)}(2q-2{\rm Re}(s)) (q-\overline{s})^{*(n+1)}
$$
$$
+(-1)^{n+1} (n+1)!  (q^2-2{\rm Re}(s)q+|s|^2)^{-(n+1)} (q-\overline{s})^{*n}
$$
$$
=(-1)^{n+2} (n+1)!  (q^2-2{\rm Re}(s)q+|s|^2)^{-(n+2)}
$$
$$
\times
[(2q-2{\rm Re}(s))(q-\bar s)-(q^2-2{\rm Re}(s)q+|s|^2)]*(q-\bar s)^{*n},
$$
where we have used the fact that the regular product
coincides with the usual one when the coefficients a real numbers. Therefore
$$
\partial^{n+1}_{x}S^{-1}(s,x)=(-1)^{n+2} (n+1)!
$$
$$
\times
 (q^2-2{\rm Re}(s)q+|s|^2)^{-(n+2)}[q^2-2q\bar s+\bar s^2]*(q-\bar s)^{*n}.
$$
The last equality in (\ref{quattordici}) depends on the fact
that $S^{-1}(s,q)(q-\overline{s})^{-1}=(q^2-2{\rm Re}(s)q+|s|^2)^{-1}$.
\end{proof}

\begin{remark}
{\rm
The proposition above provides an alternative way to prove that, on circular s-domains,
a regular function $f$ is infinitely differentiable, see Corollary \ref{hartogs}.}
\end{remark}
Provided the importance of the regular Cauchy kernel that we have introduced,
we conclude the paper by presenting
two different explicit series expansions.

\begin{theorem}\label{bounSres}
Let $q$ and $s=u+vI$ ($I\in\mathbb{S}$, $v>0$) be two quaternions such that
\begin{equation}\label{Imessequat}
v<| q-u|.
\end{equation}
Then the noncommutative Cauchy kernel can be represented by the series
\begin{equation}\label{chiusoSquat}
S^{-1}(s,q)=\sum_{n\geq 0}(q-u)^{-n-1}(vI)^n
\end{equation}

\end{theorem}
\begin{proof}
Consider the equalities
$$
(q^2-2q u+u^2+v^2)^{-1}
$$
$$
=
\Big( (q- u)^2+v^2\Big)^{-1}
$$
$$
=
\Big( (q- u)^2 (1+v^2(q- u)^{-2})\Big)^{-1}
$$
$$
=
  \left(1+v^2(q- u)^{-2}\right)^{-1}(q- u)^{-2}
$$
$$
=\sum_{n\geq 0}(-1)^n v^{2n}(q- u)^{-2n-2}.
$$
We now multiply the last expression by $(q-u+vI)$ on the right hand side and obtain:
$$
S^{-1}(s,q)=
\sum_{n\geq 0}(-1)^{n}
v^{2n}(q- u)^{-2n-2}
(q- u+vI)
$$
$$
= \sum_{n\geq 0}(-1)^{n}
v^{2n}(q-u)^{-2n-1}+
\sum_{n\geq 0}(-1)^{n}
v^{2n}(q- u)^{-2n-2}
vI
$$
Since $(-1)^nv^{2n}=(vI)^{2n}$ is a real number we obtain
$$
=
\sum_{n\geq 0}(q-u)^{-2n-1}(vI)^{2n}+
\sum_{n\geq 0}(q-u)^{-2n-2}(vI)^{2n+1}
$$
from which the statement follows.
\end{proof}
To conclude,
we now examine what happens on the complement of the closure of the domain in which
the series above  converges. We will adopt
a Laurent type approach:
\begin{theorem} Let $q$ and $s=u+vI$ ($I\in\mathbb{S}$, $v>0$) be two quaternions  such that
\begin{equation}\label{stiaa}
|q-u  |<v.
\end{equation}
Then the $S$-resolvent operator admits the series expansion:
\begin{equation}\label{stiaaaa}
S^{-1}(s,q)=\sum_{n\geq 0} (q-u)^{n}
(vI)^{-n-1}.
\end{equation}
\end{theorem}

\begin{proof}
We have the following equalities:
$$
(q^2-2qu+u^2+v^2)^{-1}
=[(q-u)^2+v^2]^{-1}= v^{-2}[1+(q-u)^2v^{-2}]^{-1}
$$
$$
=v^{-2}[\sum_{n\geq 0}(-1)^n(q-u)^{2n}v^{-2n}].
$$
By multiplying the last equality by $(q-u+vI)$ on the right hand side, we obtain:
$$
S^{-1}(s,q)=
\sum_{n\geq 0}(-1)^n(q-u)^{2n+1}v^{-2n-2}+
\sum_{n\geq 0}(-1)^n(q-u)^{2n}v^{-2n-2}(vI).
$$
Recalling that $(-1)^nv^{2n}=(vI)^{2n}$ is a real number we get:
$$
S^{-1}(s,q)=
\sum_{n\geq 0}-(q-u)^{2n+1}(vI)^{-2n-2}+
\sum_{n\geq 0}-(q-u)^{2n}(vI)^{-2n-1}
$$
which is our statement. The series converges when
(\ref{stiaa}) holds.
\end{proof}

\end{document}